\documentclass[11pt]{amsart}
\usepackage[utf8]{inputenc}
\usepackage{fontenc}
\usepackage{amsfonts}
\usepackage{amssymb}
\usepackage{amsmath}
\usepackage{amsthm}
\usepackage{enumerate}
\usepackage{hyperref}\usepackage{enumitem}
\usepackage{mathrsfs}
\usepackage{xcolor,enumitem}
\usepackage{soul}
\usepackage[all,cmtip]{xy}
\usepackage{verbatim}

\newcommand{\Z}{\mathbb{Z}}

\newcommand{\N}{\mathbb{N}}
\newcommand{\C}{\mathbb{C}}

\newcommand{\eps}{\epsilon}

\theoremstyle{Case1}

\theoremstyle{Case2}

\newcommand{\NN}{\mathbb{N}}

\newcommand{\cstu}{\mathrm{C}^*_u}
\newcommand{\csts}{\mathrm{C}^*_s}

\newcommand{\cst}{\mathrm{C}^*}
\newtheorem*{rigprob*}{Rigidity Problem for Uniform Roe Algebras}
\newtheorem*{rigprobcorona*}{Rigidity Problem for Uniform Roe Coronas}

\newcommand{\cstr}{\mathrm{C}^*}
\newcommand{\cstar}{$\mathrm{C}^*$}

\newcommand{\cB}{\mathcal{B}}
\newcommand{\cK}{\mathcal{K}}

\newtheorem{theorem}{Theorem}[section]
\newtheorem*{theorem*}{Theorem}
\newtheorem{proposition}[theorem]{Proposition}
\newtheorem{problem}[theorem]{Problem}
\newtheorem*{proposition*}{Proposition}
\newtheorem{lemma}[theorem]{Lemma}
\newtheorem*{lemma*}{Lemma}
\newtheorem{corollary}[theorem]{Corollary}
\newtheorem*{corollary*}{Corollar}

\newtheorem*{fact*}{Fact}
\theoremstyle{definition}
\newtheorem{definition}[theorem]{Definition}
\newtheorem*{definition*}{Definition}
\newtheorem{claim}[theorem]{Claim}
\newtheorem*{claim*}{Claim}

\newtheorem*{conjecture*}{Conjecture}

\newtheorem{question}[theorem]{Question}
\newtheorem*{acknowledgments}{Acknowledgments}

\theoremstyle{remark}

\newtheorem*{example*}{Example}
\newtheorem{remark}[theorem]{Remark}
\newtheorem*{remark*}{Remark}

\newtheorem*{note*}{Note}
\newtheorem*{question*}{Question}

\DeclareMathOperator{\Sp}{Sp}

\DeclareMathOperator{\supp}{supp}

\DeclareMathOperator{\propg}{prop}

\DeclareMathOperator{\rank}{rank}

\DeclareMathOperator{\Ad}{Ad}

\newcounter{my_enumerate_counter}
\newcommand{\pushcounter}{\setcounter{my_enumerate_counter}{\value{enumi}}}
\newcommand{\popcounter}{\setcounter{enumi}{\value{my_enumerate_counter}}}

\usepackage{enumitem}

\begin{document}

\title[Coarse Baum-Connes conjecture and rigidity]{Coarse Baum-Connes conjecture and rigidity for Roe algebras}
%\author{Bruno M. Braga, Yeong Chyuan Chung \and Kang Li}

\thanks{Braga is supported by the Simons Foundations. Chung and Li are supported by the European Research Council (ERC-677120).}

\author{Bruno M. Braga}
\address[B. M. Braga]{Department of Mathematics and Statistics,
York University,
4700 Keele Street,
Toronto, Ontario, Canada, M3J
1P3}
\email{demendoncabraga@gmail.com}
\urladdr{https://sites.google.com/site/demendoncabraga}

\author{Yeong Chyuan Chung}
\address[ Y. C.  Chung]{Institute of Mathematics of the Polish Academy of Sciences, \'{S}niadeckich 8, 00-656 Warsaw, Poland}
\email{ychung@impan.pl}
\urladdr{https://sites.google.com/view/ycchung/home}
\author{Kang Li}
\address[K. Li]{Institute of Mathematics of the Polish Academy of Sciences, \'{S}niadeckich 8, 00-656 Warsaw, Poland}
\email{kli@impan.pl}
\urladdr{https://sites.google.com/site/kanglishomepage}

\subjclass[2010]{}
\thanks{}
\date{\today}%
\maketitle

\begin{abstract}
In this paper, we connect the rigidity problem and the coarse Baum-Connes conjecture for Roe algebras. In particular, we show that if $X$ and $Y$ are two uniformly locally finite metric spaces such that their Roe algebras are $*$-isomorphic, then $X$ and $Y$ are coarsely equivalent provided either $X$ or $Y$ satisfies the coarse Baum-Connes conjecture with coefficients. It is well-known that coarse embeddability into a Hilbert space implies the coarse Baum-Connes conjecture with coefficients. On the other hand, we provide a new example of a finitely generated group satisfying the coarse Baum-Connes conjecture with coefficients but which does not coarsely embed into a Hilbert space. 
\end{abstract}

\section{Introduction}

Given a metric space $(X,d)$, one defines  the uniform Roe algebra of $X$, denoted by $\cstu(X)$, as the closure of all bounded operators on $\ell_2(X)$ with finite propagation (we refer the reader to Section \ref{SectionBackground} for all the definitions in this introduction). Similarly, the Roe algebra of $X$, denoted by $\cstr(X)$, is defined as
the closure of all locally compact bounded operators on $\ell_2(X,H_X)$ with finite propagation, where $H_X$ is an infinite dimensional   separable Hilbert space. Recently, the study of rigidity properties for those \cstar-algebras has gained a lot of attention (\cite{SpakulaWillett2013AdvMath,WhiteWillett2017,BragaFarah2018,BragaFarahVignati2018,BragaFarahVignati2019,BragaVignati2019}).
Precisely, the following is open.

\begin{problem}[Rigidity Problem]\label{ProblemRigidity}
Let $X$ and $Y$ be uniformly locally finite metric spaces.
\begin{enumerate}
\item\label{ProblemRigidityUnifRoeAlg}If $\cstu(X)$ and $\cstu(Y)$ are $*$-isomorphic, are $X$ and $Y$ coarsely equivalent?
\item\label{ProblemRigidityRoeAlg} If $\cstr(X)$ and $\cstr(Y)$ are $*$-isomorphic, are $X$ and $Y$ coarsely equivalent?
\end{enumerate}
\end{problem}

The first positive partial result for Problem \ref{ProblemRigidity} was proved by J. \v{S}pakula and R. Willett, who showed that both items above have a positive answer if the spaces have property A (\cite[Theorem 4.1]{SpakulaWillett2013AdvMath}). Later, Problem \ref{ProblemRigidity}\eqref{ProblemRigidityUnifRoeAlg} was answered positively for the larger class of spaces which coarsely embed into a Hilbert space (\cite[Corollary 1.2]{BragaFarah2018}), and even more recently it was shown that one only needs to assume that one of the spaces coarsely embeds into a Hilbert space (\cite[Corollary 1.5]{BragaFarahVignati2019}). This was done by looking at a technical condition on the metric spaces (Definition \ref{DefiMainGeomProp}\eqref{ItemDefiMainGeomProp}), and noticing that coarse embeddability into a Hilbert space implies that condition. However, after the work of J. \v{S}pakula and R. Willett, the Roe algebra has been neglected in the papers mentioned above, i.e., those subsequent articles only dealt with the rigidity question  for uniform Roe algebras.

The goal of this paper is two-fold. Firstly, we generalize the main results of \cite{BragaFarah2018,BragaFarahVignati2019} in order to obtain partial answers to Problem \ref{ProblemRigidity}\eqref{ProblemRigidityRoeAlg} outside of the realm of property A. Secondly, we further study the technical geometric condition introduced in \cite{BragaFarah2018,BragaFarahVignati2018} -- i.e, the property that all sparse subspaces of a given metric space $X$ yield only compact ghost projections in their (uniform) Roe algebras (see Definition \ref{DefiMainGeomProp}) -- in order to extend the class of uniformly locally finite metric spaces satisfying this property.

We now describe our main results. Firstly, let us define the main technical  geometric condition considered in these notes. For the definition of a ghost operator in $\cstr(X)$, we refer to Definition \ref{DefiGhost} below.

\begin{definition}\label{DefiMainGeomProp}
Let   $(X,d)$ be a metric space. 
\begin{enumerate}
\item $X$ is called \emph{sparse} if there exists a partition $X=\bigsqcup_nX_n$ such that
\begin{itemize}
\item $|X_n|<\infty$ for all $n\in\N$, and
\item $d(X_n,X_m)\to \infty$ as $n+m\to \infty$.
\end{itemize}
\item \label{ItemDefiMainGeomProp} We say that \emph{all sparse subspaces of $X$ yield only compact ghost projections in their Roe algebras} if for all sparse subspaces $X'\subset X$ all ghost projections in $\cstr(X') $ are compact.
\end{enumerate}
\end{definition}

Notice  that the geometric condition in \eqref{ItemDefiMainGeomProp}   was already  known to be formally  weaker than coarse embeddability into a Hilbert space (see \cite[Lemma 7.3]{BragaFarah2018}). However, as we prove below, this is actually a strictly weaker property  -- see Theorem \ref{ThmCBCwithCoefImplyMainGeomProp} and Theorem \ref{ThmPropWhichImplyMainGeomProp}.
 
Following the terminology introduced in \cite{BragaFarah2018} for the uniform Roe algebra, some of our results below depend on \emph{rigid $*$-isomorphisms} between different kinds of Roe algebras (see Definition \ref{DefiRigid} for the precise definition). The next theorem summarizes the state of the art of Problem \ref{ProblemRigidity}\eqref{ProblemRigidityRoeAlg}.

\begin{theorem}\label{ThmMainEquivalences} 
Let $X$ and $Y$ be uniformly locally finite metric spaces. Then the following are equivalent.\footnote{Definition \ref{DefiAlgebras} below gives the precise definitions of the Roe algebra, uniform Roe algebra, stable Roe algebra, and uniform algebra -- $\cstr(X)$, $\cstu(X)$, $\csts(X)$, and $\mathrm{UC}^*(X)$, respectively.}
\begin{enumerate}
\item\label{ItemThmMainEquivalences.CE} $X$ is coarsely equivalent to $Y$.
\item\label{ItemThmMainEquivalences.StIsoR} $\csts(X)$ is rigidly $*$-isomorphic to $\csts(Y)$.
\item\label{ItemThmMainEquivalences.UCR} $\mathrm{UC}^*(X)$ is rigidly $*$-isomorphic to $\mathrm{UC}^*(X)$.
\item\label{ItemThmMainEquivalences.RoeR} $\cstr(X)$ is rigidly $*$-isomorphic to $\cstr(Y)$.
\end{enumerate}
If all sparse subspaces of $Y$ yield only compact ghost projections in their Roe algebras, then the items above are also equivalent to the following.
\begin{enumerate}\setcounter{enumi}{4}
\item\label{ItemThmMainEquivalences.Morita} $\cstu(X)$ is Morita equivalent to $\cstu(Y)$.
\item\label{ItemThmMainEquivalences.StIso} $\csts(X)$ is $*$-isomorphic to $\csts(Y)$.
\item\label{ItemThmMainEquivalences.UC} $\mathrm{UC}^*(X)$ is $*$-isomorphic to $\mathrm{UC}^*(Y)$.
\item\label{ItemThmMainEquivalences.Roe} $\cstr(X)$ is $*$-isomorphic to $\cstr(Y)$.
\end{enumerate}
\end{theorem}

The main contribution of these notes to Theorem \ref{ThmMainEquivalences} are the implications \eqref{ItemThmMainEquivalences.Morita}$\Rightarrow$\eqref{ItemThmMainEquivalences.CE},  \eqref{ItemThmMainEquivalences.StIso}$\Rightarrow$\eqref{ItemThmMainEquivalences.CE},  \eqref{ItemThmMainEquivalences.UC}$\Rightarrow$\eqref{ItemThmMainEquivalences.CE}, and 
\eqref{ItemThmMainEquivalences.Roe}$\Rightarrow$\eqref{ItemThmMainEquivalences.CE}, under the hypothesis that all sparse subspaces of $Y$ yield only compact ghost projections in their Roe algebras.

As for our second goal, we start by making  explicit a definition which was already implicit in all the rigidity papers mentioned above. 

\begin{definition}
Let $X$ be a uniformly locally finite metric space. 
\begin{enumerate}
\item We say that $X$ is \emph{Roe rigid} if $X$ is coarsely equivalent to any uniformly locally finite metric space $Y$ so that $\cstr(X) $ and $\cstr(Y) $ are $*$-isomorphic.
\item We say that $X$ is \emph{uniform Roe rigid} if $X$ is coarsely equivalent to any uniformly locally finite metric space $Y$ so that $\cstu(X)$ and $\cstu(Y)$ are $*$-isomorphic.
\end{enumerate}
\end{definition}

\iffalse
\begin{remark}
It is worth noticing that up to bijective coarse equivalence, we may always assume that a uniformly locally finite metric space is also uniformly discrete. Indeed, if $(X,d)$ is a uniformly locally finite metric space, then $\partial:=d+1$ defines a metric on $X$, making it uniformly locally finite and uniformly discrete. Moreover, the identity map is a bijective coarse equivalence between $(X,d)$ and $(X,\partial)$.
\end{remark}
\fi

With this terminology, Theorem \ref{ThmMainEquivalences} states that if all sparse subspaces of a uniformly locally finite metric space yield only compact ghost projections in their Roe algebras, then $X$ is Roe rigid. Under the same conditions, we can also conclude that $X$ is uniform Roe rigid (Proposition \ref{PropUnfRoeRig}). We use this in order to prove the following:

\begin{theorem} \label{ThmCBCwithCoefImplyMainGeomProp}
Let $X$ be a uniformly locally finite metric space and assume that $X$ satisfies the coarse Baum-Connes conjecture with coefficients. Then $X$ is both Roe rigid and uniform Roe rigid.
\end{theorem}

Since there are uniformly locally finite metric spaces which satisfy the coarse Baum-Connes conjecture with coefficients but which do not coarsely embed into a Hilbert space (see Remark \ref{RemGpNoCBC} and Proposition \ref{PropGpNoCBC}), Theorem \ref{ThmCBCwithCoefImplyMainGeomProp} provides  new examples of metric spaces for which Problems \ref{ProblemRigidity}\eqref{ProblemRigidityUnifRoeAlg} and  \ref{ProblemRigidity}\eqref{ProblemRigidityRoeAlg} have a positive answer.

Moreover, we obtain some other technical conditions on a uniformly locally finite metric space so that it is Roe rigid.  We refer the reader to Section \ref{SectionCBC} for the definitions of those technical conditions.

\begin{theorem} \label{ThmPropWhichImplyMainGeomProp}
Let $X$ be a uniformly locally finite metric space. Then $X$ is both Roe rigid and uniform Roe rigid if any one of the following conditions holds:
\begin{enumerate}
\item    $X$  is sparse, admits a fibred coarse embedding into a Hilbert space,    and satisfies the coarse Baum-Connes conjecture.

\item  $X=\Box \Gamma$ is the box space of a residually finite, finitely generated discrete group $\Gamma$ that admits a coarse embedding into a Banach space with property (H), and $X$ satisfies the coarse Baum-Connes conjecture.

\item  $X=\Gamma$ is a countable discrete group which satisfies the   Baum-Connes conjecture with coefficients.
\end{enumerate}
\end{theorem}

This paper is organized as follows. In Section \ref{SectionBackground}, we deal with the definitions and background necessary for these notes. In particular, in Subsection \ref{SubsectionCBC}, we talk about the Baum-Connes conjectures. In Section \ref{SectionRigidityHomo}, we present the main tool in order to generalize the results for uniform Roe algebras obtained in \cite{BragaFarah2018,BragaFarahVignati2019} to the context of Roe algebras (Lemma \ref{LemmaFinRankProjInBDelta}). Section \ref{SectionEmbHereSubalg} starts by dealing with embeddings between Roe algebras, which is the essential step  so that we can have the asymmetry of Theorem \ref{ThmMainEquivalences}, i.e., the fact that only $Y$ has to satisfy a geometric condition. The proof of Theorem \ref{ThmMainEquivalences} is also presented in Section \ref{SectionEmbHereSubalg}. Finally, Section \ref{SectionCBC} deals with the coarse Baum-Connes conjecture, the proof of Theorems \ref{ThmCBCwithCoefImplyMainGeomProp} and \ref{ThmPropWhichImplyMainGeomProp}, and provides examples of new spaces for which the rigidity problem has a positive answer.

\section{Preliminaries}\label{SectionBackground}

 If $H$   is a Hilbert space, we denote the closed unit ball of $H$ by $B_H$. Moreover,    $\cB(H)$ and $\cK(H)$ denote the spaces of bounded and compact operators on the Hilbert space $H$, respectively. 

Let $X$ be a set and $H_X$ be a Hilbert space. Then the Hilbert space $\ell_2(X)\otimes H_X$ is canonically isometric to $\ell_2(X,H_X)$.  Given $x,y\in X$, the operator $e_{xy}\in \cB(\ell_2(X,H_X))$ is defined by
\[e_{xy}\delta_z\otimes w=\langle \delta_z,\delta_x\rangle\delta_y\otimes w\]
for all $z\in X$ and all $w\in H_X$. Given $a\in \cB(\ell_2(X,H_X))$, we let $a_{xy}= e_{yy}ae_{xx}$ for all $x,y\in X$, and define
\[\supp(a)=\{(x,y)\in X\times X\mid a_{xy}\neq 0\}.\] Clearly, for each $x,y\in X$,  $a_{xy}$ can be canonically identified with an element of $\cB(H_X)$, and we do so without further mention throughout this paper. Given $v,u\in H_X$, the operator $e_{(x,v),(y,u)}\in \cB(\ell_2(X,H_X))$ is defined by
\[e_{(x,v),(y,u)}\delta_z\otimes w=\langle\delta_z\otimes w, \delta_x\otimes v\rangle\delta_y\otimes u\]
for all $z\in X$ and all $w\in H_X$. Given $A\subset X$, define $\chi_A=\sum_{x\in A}e_{xx}$. If $A=X$, we simply write $1=\chi_X$.

\subsection{Roe algebras} If $(X,d)$ is a metric space and $a\in \cB(\ell_2(X,H_X))$, the \emph{propagation of $a$} is defined by 
\[\propg(a)=\sup \{d(x,y)\mid e_{yy}ae_{xx}\neq 0\}.\]
We say the metric space $X$ is \emph{uniformly locally finite}, which we abbreviate by \emph{u.l.f}, if \[\sup_{x\in X}|\{y\in X\mid d(x,y)\leq r\}|<\infty\] for all $r>0$.

We now define the central object of this work, i.e., the Roe algebra of a u.l.f. metric space. We also define variants of this algebra whose rigidity properties are closely related to the ones of the Roe algebra (Theorem \ref{ThmMainEquivalences}).

\begin{definition}\label{DefiAlgebras}
Let $X$ be a u.l.f. metric space and $H_X$ be  an infinite dimensional separable  Hilbert space. 
\begin{enumerate}
\item The \emph{Roe algebra of $X$ over $H_X$}, denoted by $\cstr(X)$, is defined as the closure of all $a\in \cB(\ell_2(X,H_X))$ so that $\propg(a)<\infty$ and $a_{xy}$ is compact for all $x,y\in X$. 
\item The \emph{uniform  algebra of $X$ over $H_X$}, denoted by $\mathrm{UC}^* (X)$, is defined as the closure of all $a\in \cB(\ell_2(X,H_X))$ so that $\propg(a)<\infty$ and so that there exists $N\in\N$ such that $\rank(a_{xy})\leq N$   for all $x,y\in X$. 
\item The \emph{stable Roe algebra of $X$ over $H_X$}, denoted by $\csts(X)$, is defined as the closure of all $a\in \cB(\ell_2(X,H_X))$ so that $\propg(a)<\infty$ and so that there exists a finite dimensional subspace $H\subset H_X$  such that $a_{xy}\in \cB(H)$   for all $x,y\in X$.
\item If $H_X= \C$, all the algebras above coincide and it is called the \emph{uniform Roe algebra of $X$}, denoted by $\cstu(X)$.
\end{enumerate}
Notice that, in order to simplify notation, the space $H_X$ is omitted in the notations of the algebras above. 
\end{definition} 

Clearly, $\csts(X)\subset \mathrm{UC}^*(X)\subset \cstr(X)$. Moreover, fixing a rank 1 projection on $\cB(H_X)$, there is a canonical embedding of $\cstu(X)$   into $\csts(X)$ for all u.l.f. metric spaces $X$. Also, notice that $\csts(X)$ is canonically isomorphic to $\cstu(X)\otimes \cK(H_X)$ for all metric spaces $X$.

The concept of ghost operators plays an essential role in these notes.

\begin{definition}\label{DefiGhost}
Let $X$ be a u.l.f. metric space. An operator $a\in \cstr(X)$ is a \emph{ghost} if for all $\eps>0$ there exists a finite $A\subset X$ so that $\|a_{xy}\|<\eps$ for all $x,y\not\in A$. 
\end{definition}

Notice that, given any u.l.f. metric space $X$,   all compact operators on $\ell_2(X)$  are ghosts. Also,   under the canonical embeddings $\cstu(X)\hookrightarrow \csts(X)$ described above, the concept of a ghost operator in $\cstu(X)$ is also well-defined for all u.l.f. metric space $X$.

Since we are interested in several types of ``Roe algebras'' -- $\cstu(X)$, $\csts(X)$, $\mathrm{UC}^*(X)$, and $\cstr(X)$ -- the following definition has the purpose of simplifying the statements of our technical lemmas below. 

\begin{definition} 
Let $X$ be a u.l.f. metric space. A \cstar-subalgebra $A\subset \cstr(X)$ is called \emph{Roe-like} if $\csts(X)\subset A$. 
\end{definition}

Following \cite{BragaFarah2018}, we introduce the notion of rigid $*$-homomorphisms and rigid $*$-isomorphisms between Roe-like algebras. Before stating its technical definition,  let us motivate it. We are interested in when a $*$-isomorphism $ \cst(X)\to \cst(Y)$ gives a coarse equivalence $X\to Y$. Ideally, this coarse equivalence   should be related to the $*$-isomorphism and ``close to witnessing it''. More precisely, say $H=H_X=H_Y$  and let $f:X\to Y$ be a bijective coarse equivalence. Then $f$ induces a unitary  $U:\ell_2(X,H)\to \ell_2(Y,H)$ by letting  $U\delta_x\otimes u=\delta_{f(x)}\otimes u$ for all $x\in X$ and all $u\in H$. Moreover,   $\Phi=\Ad(U)$ is a $*$-isomorphism between $\cst(X)$ and $\cst(Y)$, and  this $*$-isomorphism is so that 
\[\|\Phi(e_{(x,u),(x,u)})\delta_{f(x)}\otimes u\|=1, \ \text{ for all }\ x\in X\ \text{ and all }u\in H.\]
Obviously, one should not expect that all $*$-isomorphisms $\cst(X)\to \cst(Y)$ satisfy this property above. However, the notion of rigid $*$-isomorphism introduced below is a weakening of this property which is strong enough for our goals. Moreover, as shown in Corollary \ref{CorHomoRigid}, this property is often satisfied.

\begin{definition}\label{DefiRigid} Let $X$ and $Y$ be metric spaces, and $A\subset \cstr(X)$ and $B\subset\cstr(Y)$ be Roe-like \cstar-subalgebras.   A $*$-homomorphism  $ \Phi:A\to B$ is said to be a \emph{rigid $*$-homomorphism} if
\[\sup_{u\in B_{H_X}}\inf_{x\in X}\sup_{y\in Y,v\in B_{H_Y}}\|\Phi(e_{(x,u),(x,u)})\delta_{y}\otimes v\|>0.\] 
A $*$-isomorphism  $ \Phi:A\to B$ is called a \emph{rigid $*$-isomorphism} if both $\Phi$ and $\Phi^{-1}$ are rigid $*$-homomorphisms. In this case,  the algebras $A$ and $B$ are \emph{called rigidly $*$-isomorphic}.
\end{definition}

In the definition above, an assignment $x\in X\mapsto(y_x,v_x)\in Y\times B_{H_Y}$ \emph{witnesses that $\Phi$} is rigid if there exist $u\in B_{H_X}$ and $\delta>0$  such that 
\[\|\Phi(e_{(x,u),(x,u)})\delta_{y_x}\otimes v_x\|>\delta,\]
for all $x\in X$.
\subsection{Coarse geometry}

Let $(X,d)$ and $(Y,\partial)$ be metric spaces, and $f:X\to Y$ be a map. We say that $f$ is \emph{coarse} if for all $s>0$ there exists $r>0$ so that
\[d(x,y)<s\text{ implies }\partial(f(x),f(y))<r\text{ for all } x,y\in X,\]  and we say that $f$ is \emph{expanding} if for all $ r>0$ there exists $s>0$ so that
\[d(x,y)>s\text{ implies }\partial(f(x),f(y))>r\text{ for all  }x,y\in X.\] If $f$ is both coarse and expanding, $f$ is said to be a \emph{coarse embedding}. If $f$ is a coarse embedding so that $\sup_{y\in Y}\partial (y,f(X))<\infty$, then $f$ is called a \emph{coarse equivalence}. It is easy to check that $f$ is a coarse equivalence if and only if there exists a coarse map $g:Y\to X$ so that 
\[\sup_{x\in X}d(x,g\circ f(x))<\infty\text{ and }
\sup_{y\in Y}d(y,f\circ g(y))<\infty. \]
In this case, we say that $X$ and $Y$ are \emph{coarsely equivalent}.

\subsection{Baum-Connes conjectures}\label{SubsectionCBC}
The coarse Baum-Connes conjectures are defined in terms of certain assembly maps. Due to the high technicality of the precise statements of those  conjectures and since we do not make explicit use of the description of those assembly maps, we will not present the formal definitions in this paper but only direct the reader to appropriate sources.   

Let $G$ be a  locally compact, $\sigma$-compact, Hausdorff groupoid  with a Haar system.\footnote{See \cite[Section 1]{Tu00} for the definition and main properties of groupoids and Haar systems.} 
%Loosely speaking, the Baum-Connes conjecture for $G$ states that a certain assembly map
%\[ \mu_r:K_*^{top}(G)\rightarrow K_*(\cstr_r(G)) \]
%is an isomorphism, thereby providing means of calculating the $K$-theory groups of the reduced $\cstr$-algebra of $G$. {\color{red} Can you ad a precise reference for the formal statement of this conjecture here?}
A $\cstr$-algebra endowed with an action of $G$ is called a \emph{$G$-$\cstr$-algebra}. Loosely speaking, given a $G$-$\cstr$-algebra $A$, the Baum-Connes conjecture for $G$ with coefficients in $A$ states that a certain assembly map
\[ \mu_{r,A}:K_*^{top}(G,A)\rightarrow K_*(A\rtimes_r G) \]
is an isomorphism (we refer the reader to \cite[D\'{e}finition 5.2]{Tu99hyp} for details). Therefore, this provides means of calculating the $K$-theory groups of the reduced crossed product of $A$ by $G$.

\begin{definition}\label{DefiBCor} Let $G$ be a locally compact, $\sigma$-compact, Hausdorff groupoid with a Haar system.
\begin{enumerate}
\item  Let $A$ be a $G$-$\cstr$-algebra.
We say that $G$ satisfies the \emph{Baum-Connes conjecture for $A$} if the assembly map $\mu_{r,A}$ above is an isomorphism.

\item We say that $G$ satisfies the \emph{Baum-Connes conjecture with coefficients} if the assembly map $\mu_{r,A}$ above is an isomorphism for all $G$-$\cstr$-algebras $A$.

%\item Let $\Gamma$ be a countable discrete group, and let $A$ be a $\Gamma$-$\cstr$-algebra. We say that $\Gamma$ satisfies the \emph{group Baum-Connes conjecture for $A$} if $\Gamma$ satisfies the groupoid Baum-Connes conjecture for $A$ when $\Gamma$ is regarded as a groupoid. (Add footnote saying this agrees with the original formulation of the Baum-Connes conjecture for groups)
%
%\item We say that a countable discrete group $\Gamma$ satisfies the \emph{group Baum-Connes conjecture with coefficients} if the assembly map is an isomorphism for all $\Gamma$-$\cstr$-algebras $A$.
\end{enumerate}
\end{definition}

\begin{remark}
Any countable discrete group $\Gamma$ can be regarded as a groupoid satisfying  the properties listed in  Definition \ref{DefiBCor}. Hence,   it makes sense to say that $\Gamma$ satisfies the Baum-Connes conjecture with coefficients. Moreover, this agrees with the original formulation of the Baum-Connes conjecture for groups in \cite[Conjecture 9.6]{BCH}.\footnote{We refer the reader to \cite{Tu00} for a survey of the Baum-Connes conjecture for groupoids, and to \cite{BCsurvey} for a survey of the Baum-Connes conjecture for groups.
We also refer the reader to \cite{WOBook} for facts about $K$-theory of $\cstr$-algebras. 
}
\end{remark}

In these notes, we will need the coarse Baum-Connes conjecture with coefficients for a u.l.f. metric space, which can be defined in terms of the   Baum-Connes conjecture with coefficients for a certain groupoid related to the metric space.
Precisely, let us first recall the coarse groupoid $G(X)$ associated with a u.l.f. metric space $(X,d)$. For every $R>0$, we consider the \emph{$R$-neighborhood} of the diagonal in $X\times X$, i.e., $\Delta_R:=\{(x,y)\in X\times X: d(x,y)\leq R\}$. Define
\begin{align*}
G(X):=\bigcup_{R>0}\overline{\Delta}_R\subseteq \beta(X\times X),
\end{align*}
where $\beta(X\times X)$ denotes the Stone-\v{C}ech compactification of $X\times X$.
It turns out that the domain, range, inversion and multiplication maps on the pair groupoid $X\times X$ have unique continuous extensions to $G(X)$. With respect to these extensions, $G(X)$ becomes a principal, étale, locally compact, $\sigma$-compact Hausdorff topological groupoid whose  unit space $G(X)^{(0)}$ is $\beta X$ (see \cite[Proposition~3.2]{SkandalisTuYu2002} or \cite[Theorem~10.20]{RoeBook}). Since $G(X)^{(0)}=\beta X$ is totally disconnected, $G(X)$ is also ample (see \cite[Proposition 4.1]{Exel2010}). Moreover, there is a canonical isomorphism between  $\cstu(X)$  and the reduced groupoid $\cstr$-algebra of $G(X)$ (see \cite[Proposition~10.29]{RoeBook} for a proof).

Recall that a metric space $(X,d)$ is \emph{uniformly discrete} if $\inf_{x\neq y}d(x,y)>0$.

\begin{definition}[\cite{MR2891024, MR3197659}]
Let $X$ be a uniformly discrete u.l.f. metric space. 
\begin{enumerate}
\item  $X$ satisfies the \emph{coarse Baum-Connes conjecture}   if the coarse groupoid $G(X)$ of $X$ satisfies the  Baum-Connes conjecture for $\ell_\infty(X,\mathcal{K}(\ell_2(\NN)))$.
 
\item  $X$ satisfies the   \emph{coarse Baum-Connes conjecture with coefficients} if $G(X)$ satisfies the  Baum-Connes conjecture with coefficients.

\item $X$ satisfies the \emph{boundary coarse Baum-Connes conjecture} if the boundary groupoid $G(X)|_{\beta X\backslash X}$ of $X$ satisfies the   Baum-Connes conjecture for $ \ell_\infty(X,\mathcal{K}(\ell_2(\NN)))/c_0(X,\mathcal{K}(\ell_2(\NN)))$. 
\end{enumerate}
If $X$ is a u.l.f. metric space, we say that $X$ satisfies the \emph{coarse Baum-Connes conjecture}  (resp. \emph{coarse Baum-Connes conjecture with coefficients}, or  \emph{boundary coarse Baum-Connes conjecture}) if $X$ is bijectively coarsely equivalent to a uniformly discrete metric space which satisfies the \emph{coarse Baum-Connes conjecture}  (resp. coarse Baum-Connes conjecture with coefficients, or  boundary coarse Baum-Connes conjecture).\footnote{Notice that if $X$ is bijectively coarsely equivalent to a metric space satisfying one of those properties, then all metric spaces which are bijectively coarsely equivalent to $X$ also satisfy the same property. } 
\end{definition}

\begin{remark}
In part (2) of the definition above, we require the groupoid $G(X)$ to satisfy the Baum-Connes conjecture for all coefficients but in fact, this is equivalent to requiring $G(X)$ to satisfy the Baum-Connes conjecture for all separable coefficients.

Indeed, if $G$ is a second countable, \'{e}tale, ample groupoid, and $A$ is a $G$-$\cstr$-algebra, then $A$ may be written as an inductive limit $\varinjlim_{i \in I}(A_i, \phi_i)$ of separable $G$-$\cstr$-algebras, where each connecting homomorphism $\phi_i$ is injective. Hence, $A\rtimes_r G\cong \varinjlim_{i\in I}A_i\rtimes_r G$, and $\lim_i K_*^{top}(G,A_i)\cong K_*^{top}(G,A)$ by \cite[Lemma 5.1 and Theorem 5.2]{BonickeDellAiera19}\footnote{Notice that this was proved in \cite[Lemma 5.1 and Theorem 5.2]{BonickeDellAiera19}  for a sequence of $\cstr$-algebras but the proof still works for a net of $\cstr$-algebras.}, and we also have $\lim_i K_*(A_i\rtimes_rG)\cong K_*(A\rtimes_rG)$ by \cite[Proposition 6.2.9 and Proposition 7.1.7]{WOBook}.
Moreover, the connecting homomorphisms are compatible with the Baum-Connes assembly maps.
Thus, if $G$ satisfies the Baum-Connes conjecture with coefficients in $A_i$ for all $i$, then $G$ satisfies the Baum-Connes conjecture with coefficients in $A$.

Although the coarse groupoid $G(X)$ is not second countable in general,   \cite[Lemma 3.3]{SkandalisTuYu2002} gives a  second countable, ample groupoid $G'$ such that $G(X)=\beta X\rtimes G'$.
Then any $G(X)$-$\cstr$-algebra $A$ is also a $G'$-$\cstr$-algebra in a canonical way, and $A\rtimes_r G(X)=A\rtimes_r (\beta X\rtimes G')\cong A\rtimes_r G'$.
By \cite[Lemma 4.1]{SkandalisTuYu2002}, the assembly map for $G(X)$ with coefficients in $A$ agrees with the assembly map for $G'$ with coefficients in $A$.
\end{remark}

In Section \ref{SectionCBC}, we will make use of some well-known results about the various Baum-Connes conjectures, and we summarize them here for ease of reference.

\begin{theorem} \label{BCconjectures} \leavevmode
\begin{enumerate}
\item\label{ItemBCconjectures.subsp} If a uniformly discrete, u.l.f. metric space $X$ satisfies the coarse Baum-Connes conjecture with coefficients, then so does any subspace of $X$ \cite[Theorem 4.2]{MR2891024}.

\item\label{ItemBCconjectures.subgpoid} Let $G$ be a locally compact groupoid isomorphic to $X\rtimes G'$, where $X$ is a compact space and $G'$ is a locally compact, second countable and \'{e}tale groupoid. Let $H$ be a closed, \'{e}tale subgroupoid of $G$. If $G$ satisfies the   Baum-Connes conjecture with coefficients, then so does $H$ \cite[Theorem 3.14]{MR2891024}.

\item\label{ItemBCconjectures.FCE} If a uniformly discrete u.l.f. metric space $X$ admits a fibred coarse embedding into Hilbert space, then the boundary groupoid $G(X)|_{\beta X\backslash X}$ is a-T-menable (\cite[Theorem 1]{FinnSell2014}). Hence, $X$ satisfies the boundary coarse Baum-Connes conjecture \cite[Th\'{e}or\`{e}me 9.3]{Tu99moy}.

\item\label{ItemBCconjectures.gpBC} A uniformly discrete, u.l.f. metric space $X$ admits a coarse embedding into Hilbert space if and only if $G(X)$ is a-T-menable \cite[Theorem 5.4]{SkandalisTuYu2002}. Hence, if $X$ admits a coarse embedding into Hilbert space, then $X$ satisfies the coarse Baum-Connes conjecture with coefficients \cite[Th\'{e}or\`{e}me 9.3]{Tu99moy}. However, the converse is false (see Remark \ref{RemGpNoCBC} and Proposition \ref{PropGpNoCBC}).

\item\label{ItemBCconjectures.gpBC} For a countable discrete group $\Gamma$ with a proper left-invariant metric, write $|\Gamma|$ for the underlying metric space. Then $G(|\Gamma|)=\beta|\Gamma|\rtimes\Gamma$ \cite[Proposition 3.4]{SkandalisTuYu2002}. Moreover, if $\Gamma$ satisfies the   Baum-Connes conjecture with coefficients, then $|\Gamma|$ satisfies the coarse Baum-Connes conjecture with coefficients \cite[Lemma 4.1]{SkandalisTuYu2002}.
\end{enumerate}
\end{theorem}

\begin{remark} \label{RemGpNoCBC}
In \cite[Theorem~1.2]{AT18}, G. Arzhantseva and R. Tessera provide an example of a finitely generated group $\Lambda$ which is a split extension of an (infinite rank) abelian group by a finitely generated group with the Haagerup property such that $\Lambda$ does not coarsely embed into a Hilbert space. However, $\Lambda$ satisfies the   Baum-Connes conjecture with coefficients (see \cite[Corollary~3.14]{MR1836047} and \cite[Theorem~1.1]{MR1821144}).
\end{remark}

Besides the example in Remark \ref{RemGpNoCBC} above,  the authors of   \cite{AT18} also provide an additional example of a finitely generated group $\Gamma$ which is a split extension of a finitely generated group with property A by a finitely generated group with the Haagerup property such that $\Gamma$ does not coarsely embed into a Hilbert space (see  \cite[Theorem~1.3]{AT18}). More precisely, $\Gamma=\Z/2\Z\wr_G (H\times F)$ is a restricted wreath product, where
\begin{itemize}
\item $G$ is a finitely generated group which contains an expander graph isometrically in its Cayley graph (see \cite[Theorem~4]{Osaj14});
\item $H$ is a finitely generated group with the Haagerup property but it does not have property A (see \cite[Theorem~2]{Osaj14}) and there is a surjective homomorphism $H\twoheadrightarrow G$ (see \cite[Proposition~2.15]{AT18});
\item $F$ is a finitely generated free group with a surjective homomorphism $F\twoheadrightarrow G$;
\item $G$ is the $H\times F$-set, where $H$ acts by the left translation and $F$ acts by the right translation via their surjections onto $G$.
\end{itemize}

To the  best of our knowledge, it was unknown whether $\Gamma$ satisfies the   Baum-Connes conjecture with coefficients. We provide an affirmative answer here even though $\Gamma$ does not coarsely embed into a Hilbert space:

\begin{proposition} \label{PropGpNoCBC}
The group $\Gamma$  above satisfies the  Baum-Connes conjecture with coefficients. In particular, $|\Gamma|$ satisfies the coarse Baum-Connes conjecture with coefficients.
\end{proposition}
\begin{proof}
Since the $H$-action and the $F$-action commute, we have that $\Gamma=(\Z/2\Z\wr_G H)\rtimes F$. Observe that $\Z/2\Z\wr_G H$ satisfies the   Baum-Connes conjecture with coefficients as it is an extension of two groups with the Haagerup property (see \cite[Corollary~3.14]{MR1836047} and \cite[Theorem~1.1]{MR1821144}).

Since the free group $F$ is torsion-free and has the Haagerup property, we conclude that $\Gamma$ satisfies the   Baum-Connes conjecture with coefficients as well (see \cite[Theorem~7.1]{MR1817505}). The last statement follows from Theorem~\ref{BCconjectures}\eqref{ItemBCconjectures.gpBC}.
\end{proof}

\section{Rigidity of  homomorphisms}\label{SectionRigidityHomo}

 In this section, we show that, under our main geometric condition, strongly continuous compact preserving $*$-homomorphisms between Roe-like algebras are rigid (see Corollary \ref{CorHomoRigid}). 
 
 \begin{quote} From now on, in order to simplify statements, every time metric spaces $X$ and $Y$ are mentioned,   infinite dimensional separable Hilbert spaces $H_X$ and $H_Y$  over which the Roe algebras $\cstr(X)$ and $\cstr(Y)$ are defined will be automatically implicitly considered.
 \end{quote}
 
 The next result is a version of \cite[Proposition 4.1]{BragaFarahVignati2018} for Roe-like algebras, and it is our main tool in order to prove rigidity of such $*$-homomorphisms.

\begin{lemma}\label{LemmaFinRankProjInBDelta}
Let  $Y$ be a  u.l.f.  metric space  and assume that all sparse subspaces of $Y$ yield only compact ghost projections in their Roe algebras.     Let $(p_n)_n$ be an orthogonal  sequence of nonzero finite rank projections such that  $\text{SOT-}\sum_{n\in M }p_n\in \cstr(Y)$ for all $M\subset \N$. Then
\[ \inf_{n\in \N}\sup_{y\in Y,v\in B_{H_Y}}\|p_n\delta_y\otimes v\|>0.\]
\end{lemma}

\begin{proof}
If this fails, by going to a subsequence, we assume that $\|p_n\delta_y\otimes v\|<2^{-n}$ for all $n\in \N$, all $y\in Y$ and   all $v\in B_{H_Y}$.

 %all unit vectors $v \in H_Y$.  

\begin{claim}\label{Claim1}
For each finite subset $F\subset Y$, $\lim_n\chi_Fp_n=0$.
\end{claim}
\begin{proof}
  Clearly, $\|\chi_Fp_n\|=\|p_n\chi_F\|=\sup_{\xi\in B_{\ell_2(F,H_Y)}} \|p_n\xi\|.$
Given any $\xi \in \ell_2(F,H_Y)$ with $\|\xi\|_2\leq 1$, we can always write $\xi=\sum_{y\in F}\delta_y\otimes v_y$ such that each $v_y$ belongs to $B_{H_Y}$. In particular, $\|p_n\xi\|\leq \sum_{y\in F}\|p_n\delta_y\otimes v_y\|< |F|\cdot 2^{-n}$. Hence, we conclude that $\|\chi_Fp_n\|<|F|\cdot 2^{-n}\to 0$, as desired.
\end{proof}
 The next claim should be compared with \cite[Claim 4.2]{BragaFarahVignati2018}. Let $\partial$ denote the metric of $Y$.

\begin{claim}\label{Claim2}
By going to a subsequence of $(p_n)_n$, there exists a sequence $(Y_n)_n$ of disjoint finite non-empty subsets of $Y$ and a sequence of finite rank projections $(q_n)_n$ in $\cstr(Y)$ such that 
\begin{enumerate}
\item $\partial(Y_k,Y_m)\to \infty$ as $k+m\to \infty$ and $k\neq m$, 
\item $\|p_n-q_n\|<2^{-n}$, and 
\item $q_n\in \cK(\ell_2(Y_n, H_Y))$, for all $n\in\N$.
\end{enumerate}
\end{claim} 
 
 \begin{proof}
We construct sequences  $(q_k)_k$, $(Y_k)_k$  and $(n_k)_k$ by induction as follows. Since $p_1$ has finite rank, pick  a finite rank projection $q_1\in \cstr(Y)$ with finite support 
  such that $\|p_1-q_1\|< 2^{-1}$ and set $n_1=1$.   Pick a finite   $Y_1\subseteq   Y$ so that $\supp(q_1)\subseteq   Y_1\times Y_1$. Fix $k>1$ and assume that $Y_j$, $n_j$ and $q_j$ have been defined for all $j\leq k-1$.

Let \[Z=\Big\{y\in Y\mid \partial \Big(y,\bigcup_{j\leq k-1} Y_j\Big)\leq k\Big\}.\] Since $Y$ is uniformly locally finite,   $Z$ is finite. 
By Claim \ref{Claim1}, for all large enough $m$ we have $\|\chi_Z p_m\| <2^{-k-2}$. 
Fix such $m$. For a sufficiently large finite $Y_k\subseteq Y\setminus Z$, the operator $a= \chi_{Y_k} p_m \chi_{Y_k}$
 is a positive contraction in $\cstr(Y)$ and it satisfies $\|a-p_m\|<2^{-k-1}$. Hence $\Sp(a)\subset [0,1/2)\cup(1/2,1]$, so the map $f\colon \Sp(a)\to \{0,1\}$ defined by $f(t)=0$ if $t<1/2$ and $f(t)=1$ if $t>1/2$ is continuous. By the continuous functional calculus, $q_k=f(a)$ is a projection and $\|q_k-a\|<2^{-k-1}$ (cf. \cite[Claim 4.2]{BragaFarahVignati2018}).  
Therefore $\|q_k-p_m\|<2^{-k}$ and $q_k\in \cst(a)\subseteq \cK(\ell_2(Y_k,H_Y))$. Hence,  $\supp(q_k)\subseteq Y_k\times Y_k$  as required. Let $n_k=m$. This completes the definition of  $(q_k)_k$, $(Y_k)_k$ and $(n_k)_k$. 
\end{proof}

Let $(Y_n)_n$ and $(q_n)_n$ be given by Claim \ref{Claim2}. Let $Y'=\bigsqcup_nY_n$, so $Y'$ is a sparse subspace of $Y$. Clearly, $\sum_nq_n$ converges in the strong operator topology to a noncompact   projection  $q\in \cB(\ell_2(Y',H_Y))$. By our choice of $(q_n)_n$, $\sum_n(q_n-p_n)$ converges in norm, so $\sum_n(q_n-p_n)\in \cstr(Y)$. Hence, $q= \sum_np_n+\sum_n(q_n-p_n)\in \cstr(Y)$, and as  $q\in \cB(\ell_2(Y',H_Y))$, we have $q\in \cstr(Y')$. 

The next claim is \cite[Claim 4 in proof of Theorem 6.1]{BragaFarah2018}, so we omit its proof.

\begin{claim}
$\sum_np_n$ is a ghost.\qed
\end{claim}

Since $\|p_n-q_n\|<2^{-n}$ for all $n$, this shows that  $q$ is also a ghost; contradiction.
\end{proof}

\begin{corollary}\label{CorHomoRigid}
Let $X$ and $Y$ be u.l.f. metric spaces  and assume that all sparse subspaces of $ Y$ yield only compact ghost projections in their Roe algebras. Let $A\subset \cstr(X)$  be a Roe-like \cstar-subalgebra.  Then every strongly continuous compact preserving $*$-homomorphism $\Phi: A\to \cstr(Y)$ is a rigid $*$-homomorphism. 
\end{corollary}

\begin{proof}
Fix a unit vector $u\in H_X$ and for each $x\in X$, let $p_x=\Phi(e_{(x,u),(x,u)})$. Then each $p_x$ is a projection and, as $\Phi$ is compact preserving, each $p_x$ has finite rank. Moreover, $(p_x)_{x\in X}$ is an  orthogonal sequence. Since $\Phi$ is strongly continuous, $\text{SOT-}\sum_{n\in M}p_n\in \cstr(Y)$ for all $M\subset \N$. Let $x\in X\mapsto (y_x,v_x)\in Y\times B_{H_Y}$ be given by Lemma \ref{LemmaFinRankProjInBDelta}, i.e., an assignment satisfying  \[\inf_{x\in X}\|\Phi(e_{(x,u),(x,u)})\delta_{y_x}\otimes v_x\|> 0.\] This assignment witnesses that  $\Phi$ is a rigid $*$-homomorphism. 
\end{proof}

\section{Embeddings onto hereditary subalgebras and isomorphisms}\label{SectionEmbHereSubalg}
 
Notice that Theorem \ref{ThmMainEquivalences} is asymmetric, since it imposes a geometric condition only on one of the metric spaces. In order to obtain this asymmetric result, we start this section by studying   embeddings between Roe-like algebras onto hereditary subalgebras. After providing  results on embedding of hereditary subalgebras, we provide a proof for Theorem \ref{ThmMainEquivalences}.
 
The next lemma follows completely analogously to \cite[Lemma 6.1]{BragaFarahVignati2019}, so we omit its proof. 

\begin{lemma}\label{LemmaPhiStronglyContAndU}
Let $X$ and $Y$ be  metric spaces, $A\subset \cstr(X)$ and $B\subset \cstr(Y)$ be Roe-like  \cstar-subalgebras, and $\Phi\colon A\to B$ be an embedding onto a hereditary \cstar-subalgebra of $B$. Then 
\[
\Phi(\cK(\ell_2(X,H_X)))=\cK(\ell_2(Y,H_Y))\cap \Phi(A).
\]
Moreover, there exists an  isometry $U\colon \ell_2(X,H_X)\to \ell_2(Y,H_Y)$ such that $\Phi(a)=UaU^*$ for all $a\in A$. In particular,  $\Phi$ is strongly continuous and rank preserving. 
\end{lemma}

The next lemma should be compared with \cite[Lemma 6.2]{BragaFarahVignati2019}.

\begin{lemma}\label{LemmaTheMapsAreExpanding} 
Suppose $X$ and $Y$ are u.l.f. metric spaces, and $A\subset \cstr(X)$ and $B\subset \cstr(Y)$ are Roe-like \cstar-subalgebras so that either $B=\csts(Y)$ or $\mathrm{UC}^*(Y)\subset B$. Let $\Phi\colon A\to B$ be a rigid embedding onto a hereditary \cstar-subalgebra of $B$, and  $x\in X\mapsto (y_x,v_x)\in Y\times B_{H_Y}$ be an assignment witnessing that $\Phi$ is a rigid $*$-homomorphism. Then the map $x\in X\mapsto y_x\in Y$ is  expanding.
\end{lemma}

\begin{proof}
Let $d$ and $\partial $ be the metrics of $X$ and $Y$ respectively. First assume $\mathrm{UC}^*(Y)\subset B$. Fix a unit vector $u\in H_X$ and $\delta>0$ such that \[\|\Phi(e_{(x,u),(x,u)})\delta_{y_x}\otimes v_x\|\geq \delta\] for all $x\in X$.  Suppose $x\in X\mapsto y_x\in Y$ is not expanding. Then there exist 
 $r>0$,  and sequences $(x^1_n)_n$ and $(x^2_n)_n$ in $X$ such that $d(x_n^1,x_n^2)\geq n$ and $\partial(y_n^1,y_n^2)\leq r$ for all $n\in\N$, where $y^1_n=y_{x^1_n}$ and $y^2_n=y_{x^2_n}$ for all $n\in\N$. To simplify notation, we also  let $v^1_n=v_{x^1_n}$ and $v^2_n=v_{x^2_n}$ for all $n\in\N$.

Since $d(x_n^1,x_n^2)\geq n$ for all $n\in\N$, by going to a subsequence, we can assume that either $(x^1_n)_n$ or $(x^2_n)_n$ is a   sequence of distinct elements. Without loss of generality, assume that this is the case for   $(x^1_n)_n$.  Moreover, going to a subsequence, we can assume that both $(y^1_n)_n$ and $(y^2_n)_n$ are sequences of distinct elements (cf. \cite[Claim 6.3]{BragaFarahVignati2019}).
 
By Lemma \ref{LemmaPhiStronglyContAndU}, $\Phi$ is rank preserving, so $(\Phi(e_{(x^1_n,u),(x^1_n,u)}))_n$ is an orthogonal sequence of   rank 1 projections. Hence, by going to a further subsequence, assume that 
\[\|e_{(y^1_n,v^1_n),(y^2_n,v^2_n)}\Phi(e_{(x^1_n,u),(x^1_n,u)})\|<  2^{-n-m-1}{\delta^2}\]
for all $n\neq m$. 

 Since $(y^1_n)_n$ and $(y^2_n)_n$ are sequences of distinct elements and $\partial(y^1_n,y^2_n)\leq r$ for all $n\in\N$, $\sum_{n\in\N}e_{(y^1_n,v^1_n),(y^2_n,v^2_n)}$ converges in the strong operator topology to an element in $\mathrm{UC}^*(Y)$, and hence, in $B$. As $\Phi(A)$ is a hereditary subalgebra of $B$, there exists $a\in A$ such that 
\[\Phi(a)=\Phi(1)\Big(\sum_{n\in\N}  e_{(y^1_n,v^1_n),(y^2_n,v^2_n)}\Big)\Phi(1).\]
Analogously as \cite[Claim 6.4]{BragaFarahVignati2019}, we have that
\[\inf_n\|e_{(x^2_n,v^2_n),(x^2_n,v^2_n)}ae_{(x^1_n,v^1_n),(x^1_n,v^1_n)}\|\geq \delta^2/2.\] Since $a\in \cstr(X)$ and $\lim_nd(x^1_n,x^2_n)=\infty$, this gives us a contradiction.

Now assume $B=\csts(Y)$. The next claim is essentially \cite[Lemma 6.4]{SpakulaWillett2013AdvMath}, so we omit its proof.

\begin{claim}
There exists a  finite rank projection $w$ on $\cB(H_Y)$  such that 
\[\inf_{x\in X}\|\Phi(e_{(x,u),(x,u)})\delta_{y_x}\otimes wv_x\|>0.\]\qed
\end{claim}

The rest of the proof is just a matter of repeating the proof for the previous case but  with this new assignment $x\in X\mapsto(y_x,wv_x)\in Y\times H_Y$. Proceeding with this strategy, since $w$ has finite rank, we can guarantee that  $\sum_{n\in\N}e_{(y^1_n,wv^1_n),(y^2_n,wv^2_n)}$ converges in the strong operator topology to an element in $\csts(Y)$, and the proof works verbatim.
\end{proof}

\begin{lemma}\label{LemmaTheMapsAreCoarse} 
Suppose $X$ and $Y$ are u.l.f. metric spaces, and $A\subset \cstr(X)$ and $B\subset \cstr(Y)$ are Roe-like \cstar-subalgebras so that either $B=\csts(Y)$ or $\mathrm{UC}^*(Y)\subset B$. Let $\Phi\colon A\to B$ be a rigid embedding onto a hereditary \cstar-subalgebra of $B$, and  $x\in X\mapsto (y_x,v_x)\in Y\times B_{H_Y}$ be an assignment witnessing that $\Phi$ is a rigid $*$-homomorphism. Then the map $x\in X\mapsto y_x\in Y$ is coarse.
\end{lemma}

\begin{proof}
By Lemma \ref{LemmaPhiStronglyContAndU}, there exists an isometry $U:\ell_2(X,H_X)\to \ell_2(Y,H_Y)$ such that $\Phi(a)=UaU^*$ for all $a\in A$. Hence, since  $x\in X\mapsto (y_x,v_x)\in Y\times B_{H_Y}$ witnesses that $\Phi$ is a rigid $*$-homomorphism, there exists $\delta>0$ and a unit vector $u\in H_X$ such that $\|\Phi(e_{(x,u),(x,u)})\delta_{y_x}\otimes v_x\|>\delta$. In other words, 
\[|\langle U\delta_{x}\otimes u,\delta_{y_x}\otimes v_x\rangle|> \delta\]
for all $x\in X$. 

If $B=\csts(Y)$, it follows from  \cite[Lemma 6.5(2)]{SpakulaWillett2013AdvMath} that  $x\in X\mapsto y_x\in Y$ is coarse, while if $\mathrm{UC}^*(Y)\subset B$, it follows from  \cite[Lemma 4.5(2)]{SpakulaWillett2013AdvMath} that  $x\in X\mapsto y_x\in Y$ is coarse.\footnote{Notice that in  \cite{SpakulaWillett2013AdvMath} the  map $U$ is unitary. However, this is not necessary and the same proof holds for $U$ being an isometry.}
\end{proof}

\begin{theorem}\label{ThmEmbHerSubAlg}
Let $X$ and $Y$ be u.l.f. metric spaces and  assume that all sparse subspaces of $Y$ yield only compact ghost projections in their Roe algebras. Let $A\subset \cstr(X)$ and $B\subset \cstr(Y)$ be Roe-like \cstar-algebras such that either $B=\csts(Y)$ or $\mathrm{UC}^*(Y)\subset B$.  If $A$  embeds onto a hereditary \cstar-subalgebra of $B$, then $X$ coarsely embeds into $Y$.
\end{theorem}

\begin{proof}
Let $\Phi:A\to B$ be an embedding onto a hereditary subalgebra of $B$. By Lemma \ref{LemmaPhiStronglyContAndU}, $\Phi$ is strongly continuous and compact preserving.  By Corollary \ref{CorHomoRigid}, $\Phi$ is a rigid $*$-homomorphism. Let $x\in X\mapsto(y_x,v_x)\in Y\times B_{H_Y}$ be an assignment which witnesses that $\Phi$ is rigid. Define $f:X\to Y$ by $f(x)=y_x$ for all $x\in X$. By Lemma \ref{LemmaTheMapsAreExpanding}, $f$ is expanding and by  Lemma \ref{LemmaTheMapsAreCoarse}, $f$ is coarse. So $f$ is a coarse embedding. 
\end{proof}

\begin{corollary}\label{Cor}
Let $X$ and $Y$ be u.l.f. metric spaces and  assume that all sparse subspaces of $Y$ yield only compact ghost projections in their Roe algebras. Let $A\subset \cstr(X)$ and $B\subset \cstr(Y)$ be Roe-like \cstar-algebras such that either $B=\csts(Y)$ or $\mathrm{UC}^*(Y)\subset B$. If $A$  embeds onto a hereditary \cstar-subalgebra of $B$,  then all sparse subspaces of $ X$ yield only compact ghost projections in their Roe algebras. 
\end{corollary}

\begin{proof}
By Theorem \ref{ThmEmbHerSubAlg}, $X$ coarsely embeds into $Y$. Hence, since the property ``all sparse subspaces yield only compact ghost projections'' passes through  coarse embeddings (see \cite[Theorem 7.6 and Remark 7.8]{BragaFarahVignati2019}), the result follows.
\end{proof}

\begin{proof}[Proof of Theorem \ref{ThmMainEquivalences}]
Since two unital \cstar-algebras algebras $A$ and $B$ are
Morita equivalent if and only if they are stably $*$-isomorphic \cite[Theorem 1.2]{BGR}, \eqref{ItemThmMainEquivalences.Morita} and \eqref{ItemThmMainEquivalences.StIso} are equivalent. It was shown in \cite[Theorem 4]{BrodzkiNibloWright2007} that \eqref{ItemThmMainEquivalences.CE} implies \eqref{ItemThmMainEquivalences.Morita} even without the assumption on sparse subspaces. Moreover, it is clear from the proof of \cite[Theorem 4]{BrodzkiNibloWright2007} that \eqref{ItemThmMainEquivalences.CE} also implies \eqref{ItemThmMainEquivalences.StIsoR}, \eqref{ItemThmMainEquivalences.UCR}, and \eqref{ItemThmMainEquivalences.RoeR}.

Suppose $\Phi:\cstr(X)\to\cstr(Y)$ is a rigid $*$-isomorphism. Let $x\in X\mapsto(y_x,v_x)\in Y\times B_{H_Y}$ and $y\in Y\mapsto(x_y,u_y)\in X\times B_{H_X}$ be assignments witnessing the rigidity of $\Phi$ and $\Phi^{-1}$, respectively. Define maps $f:X\to Y$ and $g:Y\to X$ by letting $f(x)=y_x$ and $g(y)=x_y$ for all $x\in X$ and all $y\in Y$. By Lemma \ref{LemmaTheMapsAreCoarse}, both $f$ and $g$ are coarse.

By \cite[Lemma 3.1]{SpakulaWillett2013AdvMath}, there exists a unitary $U:\ell_2(X,H_X)\to \ell_2(Y,H_Y)$ such that $\Phi(a)=UaU^*$ for all $a\in \cstr(X)$ (cf. Lemma \ref{LemmaPhiStronglyContAndU}). Proceeding exactly as in the proof of \cite[Theorem 4.1]{SpakulaWillett2013AdvMath}, we have that $f\circ g$ and $g\circ f$ are close to $\mathrm{Id}_Y$ and $\mathrm{Id}_X$, respectively. So $X$ is coarsely equivalent to $Y$, and \eqref{ItemThmMainEquivalences.RoeR} implies \eqref{ItemThmMainEquivalences.CE}. The implication  \eqref{ItemThmMainEquivalences.UCR}$\Rightarrow$\eqref{ItemThmMainEquivalences.CE} follows completely analogously, and the implication \eqref{ItemThmMainEquivalences.StIsoR}$\Rightarrow$\eqref{ItemThmMainEquivalences.CE} is the same but with \cite[Theorem 6.1]{SpakulaWillett2013AdvMath} instead of \cite[Theorem 4.1]{SpakulaWillett2013AdvMath}.

Assume that all sparse subspaces of $Y$ yield only compact ghost projections in $\cstr(Y)$. By Corollary \ref{Cor}, the same holds for $X$. Hence,  it follows from Lemma \ref{LemmaPhiStronglyContAndU} and Corollary \ref{CorHomoRigid} that \eqref{ItemThmMainEquivalences.StIso} implies \eqref{ItemThmMainEquivalences.StIsoR}, that \eqref{ItemThmMainEquivalences.UC} implies \eqref{ItemThmMainEquivalences.UCR}, and that \eqref{ItemThmMainEquivalences.Roe} implies \eqref{ItemThmMainEquivalences.RoeR}, completing the proof.
\end{proof}

We finish this section with a simple remark:
\begin{remark}
Let $X$ be a u.l.f. metric space and $A\subset \cstr(X)$ be a Roe-like \cstar-algebra. We say that \emph{$X$ yields only compact ghost projections in $A$} if all ghost projections in $A$ are compact. Proceeding as in the results above, one can obtain the following:
\begin{enumerate}
\item   Items 1-7 of Theorem \ref{ThmMainEquivalences} are all equivalent under the weaker assumption that  all sparse subspaces $Y'\subset Y$ yield only compact ghost projections in $\mathrm{UC}^*(Y')$, and 
\item  Items 1-6 of Theorem \ref{ThmMainEquivalences} are all equivalent under the weaker assumption that  all sparse subspaces $Y'\subset Y$ yield only compact ghost projections in $\csts(Y')$.
\end{enumerate}
\end{remark}

\section{The coarse Baum-Connes conjecture and Roe rigidity}\label{SectionCBC}

In this section, we prove Theorem \ref{ThmCBCwithCoefImplyMainGeomProp} and Theorem \ref{ThmPropWhichImplyMainGeomProp}, which are  consequences of Theorem \ref{ThmMainEquivalences} and  Theorem \ref{geometric condition} below.

In order to be able to evoke some results in the literature,  we must recall the definition of a metric space with only finite coarse components. Let $(X,d)$ be a u.l.f. metric space and $R>0$. The \emph{Rips complex of $X$ associated to $R$} is defined as 
\[P_R(X)=\{A\subset X\mid \mathrm{diam}(A)\leq R\}\]
and we define an equivalence relation $\sim_R$ on $P_R(X)$ by setting $A\sim_R A'$ if there exists $n\in\N$ and $x_1,\ldots, x_n\in X$ so that $x_1\in A$, $x_n\in A'$ and $d(x_i,x_{i+1})\leq R$ for all $i\in \{1,\ldots, n-1\}$. We say that $X$ has \emph{only finite coarse components} if for all $R>0$ every $\sim_R$-equivalence class of $P_R(X)$ is finite.

As we see below, a u.l.f.  metric space with only finite coarse components is simply a sparse metric space in disguise. 

\begin{proposition}\label{PropSparseIFFOnlyFinCoarComp}
A u.l.f.  metric space has only finite coarse components if and only if it is sparse.
\end{proposition}

\begin{proof}
Clearly, every sparse metric space has only finite coarse components. Let $X=\{x_n\mid n\in\N\}$ be a countable metric space with metric $d$ and assume that $X$ has only finite coarse components. We construct the partition $(X_n)_n$ of $X$ which witnesses that $X$ is sparse by induction. Let $X_1$ be the union of the elements in the  $\sim_1$-equivalence class of $P_1(X)$ containing $\{x_1\}$. By hypothesis, $X_1$ is finite and $d(X_1,X\setminus X_1)>1$. Suppose $X_1,\ldots, X_{n}$ have been defined, that $X_i$ is finite for all $i\in \{1,\ldots,n\}$ and that $d(X_n,X' )>n$, where $X'=X\setminus \bigcup_{i\leq n}X_i$. Let $X''$ be the union of the elements in the $\sim_{n+1}$-equivalence class of $P_{n+1}(X)$ containing $\{x_{m}\}$, where $m=\min\{i\in\N\mid x_i\not\in \bigcup_{i\leq n}X_i\}$.  Set $X_{n+1}=X''\cap X' $. Then $X_{n+1}$ is finite and $d(X_{n+1},X' \setminus X_{n+1})>n+1$. This procedure clearly shows that $X$ is sparse.
\end{proof}

The following theorem is proved in the same way as a result of Martin Finn-Sell in \cite[Proposition~35]{FinnSell2014}:
\begin{theorem} \label{ThmBdryInj}
Let $X$ be a sparse uniformly discrete u.l.f. metric space. Assume that the boundary coarse Baum-Connes assembly map for $X$ is injective. If $[p]_0\in K_0(\cstr(X))$ is the class of a noncompact ghost projection in the Roe algebra $\cstr(X)$, then $[p]_0$ is not in the image of the coarse Baum-Connes assembly map for $X$.

\end{theorem}
\begin{proof}
By Proposition \ref{PropSparseIFFOnlyFinCoarComp}, $X$ has only finite coarse components.  The proof is now  identical to the proof of \cite[Theorem~4.6]{MR3197659} (see also the proof of \cite[Proposition~35]{FinnSell2014}). 
\end{proof}

%{\color{red} For the next theorem, we need to recall the notion of property (H) for Banach spaces.
%\begin{definition}\cite{MR2980001}
%A real Banach space $V$ is said to have \emph{property (H)} if there exists an increasing sequence of finite dimensional subspaces $\{V_n\}$ of $V$, and an increasing sequence of finite dimensional subspaces $\{W_n\}$ of a real Hilbert space such that
%\begin{enumerate}
%\item $\bigcup_nV_n$ is dense in $V$;
%\item there exists a uniformly continuous map $\psi:S(\bigcup_n V_n)\rightarrow S(\bigcup_n W_n)$ such that the restriction of $\psi$ to $S(V_n)$ is a homeomorphism onto $S(W_n)$ for each $n\in\NN$, where $S(\cdot)$ denotes the unit spheres of the respective spaces.
%\end{enumerate}
%\end{definition}

%This property was introduced in \cite{MR2980001}, where it was shown that groups that are coarsely embeddable into Banach spaces with property (H) satisfy the strong Novikov conjecture.
%Moreover, in \cite{ChenWangYu2015}, it was shown that u.l.f. metric spaces that are coarsely embeddable into Banach spaces with property (H) satisfy the coarse Novikov conjecture.
%Examples of Banach spaces with property (H) include $\ell_p(\NN)$ for $p\geq 1$, the space of Schatten $p$-class operators on a Hilbert space for $p\geq 1$, and uniformly convex Banach spaces with certain unconditional bases.}

\begin{theorem}\label{geometric condition}
Let $X$ be a uniformly discrete u.l.f. metric space. Then all sparse subspaces of $X$ yield only compact ghost projections in their Roe algebras if any of the following conditions holds: 
\begin{enumerate}
\item $X$ satisfies the coarse Baum-Connes conjecture with coefficients.

\item  $X$ is sparse, admits a fibred coarse embedding into a Hilbert space\footnote{See \cite[Definition~2.1]{MR3116568} for the definition of fibred embedding into Hilbert spaces.},  and satisfies the coarse Baum-Connes conjecture.

\item  $X=\Box \Gamma$ is any box space\footnote{See \cite[Definition 11.24]{RoeBook} for the definition of box spaces.} of a residually finite, finitely generated discrete group $\Gamma$ that admits a coarse embedding into a Banach space with property (H)\footnote{See \cite[Definition~1.1]{MR2980001} for the definition of property (H). Examples of Banach spaces with property (H) include $\ell_p(\NN)$ for $p\geq 1$, the Banach space of Schatten $p$-class operators on a Hilbert space for $p\geq 1$, and Banach spaces with nontrivial cotype admiting an unconditional bases (see \cite{{MR2980001}} and \cite{ChengWang2018} for more details).}, and $X$ satisfies the coarse Baum-Connes conjecture.

\item  $X=\Gamma$ is a countable discrete group which satisfies the   Baum-Connes conjecture with coefficients.
\end{enumerate}
\end{theorem}

\begin{proof}
(1): 
Assume that $X$ satisfies the coarse Baum-Connes conjecture with coefficients.
By Theorem \ref{BCconjectures}\eqref{ItemBCconjectures.subsp}, every sparse subspace $\tilde{X}\subset X$ also satisfies the coarse Baum-Connes conjecture with coefficients, i.e., the coarse groupoid $G(\tilde{X})$ of $\tilde{X}$ satisfies the   Baum-Connes conjecture with coefficients.

  By \cite[Lemma 3.3]{SkandalisTuYu2002}, $G(\tilde{X})=\beta\tilde{X}\rtimes G'$ for some locally compact, second countable, \'{e}tale groupoid $G'$, so by Theorem \ref{BCconjectures}\eqref{ItemBCconjectures.subgpoid}, the closed \'{e}tale subgroupoid $G(\tilde{X})|_{\beta \tilde{X}\setminus \tilde{X}}$ satisfies the   Baum-Connes conjecture with coefficients. By definition, this implies that  $\tilde{X}$ satisfies the boundary coarse Baum-Connes conjecture. By Theorem~\ref{ThmBdryInj}, there cannot be a noncompact ghost projection in $\cstr(\tilde{X})$.

(2):
Since $X$ is sparse and admits a fibred coarse embedding into a Hilbert space, $X$ satisfies the boundary coarse Baum-Connes conjecture by Theorem \ref{BCconjectures}\eqref{ItemBCconjectures.FCE}. Since $X$ also satisfies the coarse Baum-Connes conjecture by assumption, there cannot be a noncompact ghost projection in $\cstr(X)$ by Theorem \ref{ThmBdryInj}.

(3):
It follows from \cite[Proposition~2.5]{MR3197659} that $G(X)|_{\beta X\backslash X}\cong (\beta X\backslash X)\rtimes \Gamma$. Hence, the boundary coarse Baum-Connes assembly map for $X$
\begin{align*}
K^{\text{top}}_*\left(G(X)|_{\beta X\backslash X},\frac{\ell_\infty(X,\mathcal{K})}{c_0(X,\mathcal{K})}\right)\rightarrow K_*\left(\frac{\ell_\infty(X,\mathcal{K})}{c_0(X,\mathcal{K})}\rtimes_r G(X)|_{\beta X\backslash X}\right)
\end{align*}
is injective if and only if the  Baum-Connes assembly map

\begin{align*}
K^{\text{top}}_*\left(\Gamma, \frac{\ell_\infty(X,\mathcal{K})}{c_0(X,\mathcal{K})}\right)\rightarrow K_*\left(\frac{\ell_\infty(X,\mathcal{K})}{c_0(X,\mathcal{K})}\rtimes_r \Gamma\right)
\end{align*}
is injective by \cite[Lemma~4.1]{SkandalisTuYu2002}. Thus, the conclusion follows directly from \cite[Theorem~1.2]{MR2980001}.

(4):
If $\Gamma$ satisfies the Baum-Connes conjecture with coefficients, then it satisfies the coarse Baum-Connes conjecture with coefficients by Theorem \ref{BCconjectures}\eqref{ItemBCconjectures.gpBC}, reducing to case (1).

\end{proof}

Before we prove Theorem \ref{ThmCBCwithCoefImplyMainGeomProp} and Theorem \ref{ThmPropWhichImplyMainGeomProp}, we need to notice that previous results in the literature already imply that if all sparse subspaces of a u.l.f. metric space yield only compact ghost projections in their Roe algebra, then this space is uniform Roe rigid.

\begin{proposition}\label{PropUnfRoeRig}
Let $X$ be a u.l.f. metric space so that all of its sparse subspaces yield only compact ghost projections in their Roe algebras. Then $X$ is uniform Roe rigid.
\end{proposition}

\begin{proof}
Let $Y$ be a u.l.f. metric space so that $\cstu(X)$ and $\cstu(Y)$ are $*$-isomorphic. Then   $\cstu(X)\otimes \cK(H_X)$ and $\cstu(Y)\otimes \cK(H_Y)$ are $*$-isomorphic, i.e., $\csts(X)$ and $\csts(Y)$ are $*$-isomorphic. Hence, Theorem \ref{ThmMainEquivalences} gives us that $X$ and $Y$ are coarsely equivalent.
\end{proof}

\begin{proof}[Proof of Theorem \ref{ThmCBCwithCoefImplyMainGeomProp} and Theorem \ref{ThmPropWhichImplyMainGeomProp}]
First notice that we can assume without loss of generality that $X$  is  uniformly discrete. Indeed, let $d$ be the metric on $X$ and set $\partial=d+1$. The metric $\partial $  is clearly a uniformly discrete u.l.f. metric and the identity map $(X,d)\to(X,\partial)$ is a bijective coarse equivalence. Hence, the Roe algebras of $(X,d)$ and $(X,\partial)$ are canonically isomorphic. Moreover, all the properties considered in Theorem \ref{ThmCBCwithCoefImplyMainGeomProp} and Theorem \ref{ThmPropWhichImplyMainGeomProp} are also shared by $(X,\partial)$ given that $(X,d)$ satisfies them.

The results on Roe rigidity now follow straightforwardly from Theorem \ref{ThmMainEquivalences} and Theorem \ref{geometric condition}, and the results on uniform Roe rigidity follow   from Proposition \ref{PropUnfRoeRig} and Theorem \ref{geometric condition}.
\end{proof}

\begin{question}
There is a finitely generated group $\Gamma$ which contains an expander $X$ with large girth isometrically in its Cayley graph (see \cite[Theorem~4]{Osaj14}). Thus $\Gamma$ does not satisfy the coarse Baum-Connes conjecture with coefficients. Moreover, the sparse subspace $X$ of $\Gamma$ yields a noncompact ghost projection in $\cstr(X)$. Is $\Gamma$ Roe rigid?
\end{question}

 We finish this paper defining   uniform Roe  bijective rigidity and listing some related  results, some of which are already known from earlier results.

\begin{definition}
Let $X$ be a u.l.f. metric space. We say that $X$ is \emph{uniform Roe  bijectively rigid} if $X$ is   bijectively coarsely equivalent to any uniformly locally finite metric space $Y$ so that $\cstu(X)$ and $\cstu(Y)$ are $*$-isomorphic.
\end{definition}

 \begin{remark}
Let $X$ be a u.l.f. metric space. Then $X$ is uniform Roe bijectively rigid if any one of the following conditions holds:
\begin{enumerate}
\item  $X$ has property A. 
This case follows from \cite[Corollary~6.13]{WhiteWillett2017} and \cite[Theorem~1.11]{BragaFarahVignati2018}.

\item  $X$ is non-amenable and satisfies the coarse Baum-Connes conjecture with coefficients.
This case follows from Theorem~\ref{geometric condition}(1), Theorem~\ref{ThmMainEquivalences}, and \cite[Theorem~5.1]{WhiteWillett2017}.

\item  $X=\Gamma$ is a  group which satisfies the coarse Baum-Connes conjecture with coefficients. Indeed, if $\Gamma$ is amenable, then $\Gamma$ has property A, so we apply case (1). If $\Gamma$ is non-amenable as a group, then it is non-amenable as a metric space, so we apply case (2).
\end{enumerate}
\end{remark}

\begin{question}
Let $X$ be a u.l.f. metric space, which admits a coarse embedding into a Hilbert space. Is $X$ necessarily uniform Roe   bijectively rigid?
\end{question}

 \begin{acknowledgments}
Part of this paper was written while the first named author visited IMPAN in Warsaw and he is thankful for the warm welcome he received during his stay there. The third named author would like to thank Christian B\"onicke and Martin Finn-Sell for helpful discussions on the coarse Baum-Connes conjecture.
\end{acknowledgments}

\end{document}